\documentclass[sigconf]{acmart}
\usepackage{amsfonts,color,latexsym,amssymb,epsfig,rotating,array}
\usepackage{libertine} 
\usepackage{url} 

\def\BibTeX{{\rm B\kern-.05em{\sc i\kern-.025em b}\kern-.08emT\kern-.1667em\lower.7ex\hbox{E}\kern-.125emX}}

\copyrightyear{2019}
\acmYear{2019}
\setcopyright{acmlicensed}
\acmConference[ISSAC '19]{ISSAC '19: ACM International Symposium on Symbolic 
and Algebraic Computation}{July 15--18, 2019}{Beijing, China}
\acmBooktitle{ISSAC '19: ACM International Symposium on Symbolic 
and Algebraic Computation, July 15--18, 2019, Beijing, China}
\acmPrice{15.00}
\acmDOI{10.1145/1122445.1122456}
\acmISBN{978-1-4503-9999-9/18/06}

\newtheorem{lai}{Lairez's Theorem} 
 
\newtheorem{dfn}{Definition}[section]
\newtheorem{rem}[dfn]{Remark} 
\newtheorem{prop}[dfn]{Proposition}

\newtheorem{thm}[dfn]{Theorem} 
 
\newtheorem{lemma}[dfn]{Lemma}
\newtheorem{cor}[dfn]{Corollary}

\definecolor{purple}{rgb}{.5,0,.5}
\definecolor{red}{rgb}{.6,0,0} 
\definecolor{green}{rgb}{0,.5,0} 

\renewcommand{\qed}{$\blacksquare$}

\newcommand{\thth}{^{\text{\underline{th}}} }

\newcommand{\np}{{\mathbf{NP}}}

\newcommand{\gln}{\mathbb{G}\mathbb{L}_n}

\newcommand{\eps}{\varepsilon}

\newcommand{\E}{\mathbb{E}}

\renewcommand{\P}{\mathbb{P}}
\newcommand{\Q}{\mathbb{Q}}
\newcommand{\R}{\mathbb{R}}

\renewcommand{\C}{\mathbb{C}}

\newcommand{\N}{\mathbb{N}}
\newcommand{\Z}{\mathbb{Z}}

\newcommand{\Zn}{\Z^n}

\newcommand{\Rn}{\R^n}

\newcommand{\Cn}{\C^n}

\newcommand{\Cs}{\C^*}

\newcommand{\cA}{\mathcal{A}}

\newcommand{\Rsn}{{(\R^*)}^n}
\newcommand{\Csn}{{(\C^*)}^n}

\newcommand{\dia}{$\diamond$}

\begin{document} 

\title[Faster Solution to Smale's 17$\thth$ Problem I]{
A Faster Solution to Smale's 17$\thth$ Problem I: Real Binomial Systems}

\author{Grigoris Paouris}
\authornote{Partially supported by NSF grant DMS-1812240.}  
\email{grigoris@math.tamu.edu}
\affiliation{
\institution{Texas A\&{}M University}
\streetaddress{TAMU 3368}
\city{College Station}
\state{Texas}
\postcode{77843-3368}
}
\author{Kaitlyn Phillipson}
\authornote{Partially supported by NSF REU grant DMS-1460766   
and NSF grant CCF-1409020.}
\email{kphillip@stedwards.edu}
\affiliation{
\institution{St.\ Edwards University}
\streetaddress{3001 South Congress Ave.}
\city{Austin}
\state{Texas}
\postcode{78704} 
} 
\author{J.\ Maurice Rojas}
\authornote{Partially supported by NSF REU grant DMS-1757872  
and NSF grant CCF-1409020.}
\email{rojas@math.tamu.edu}
\affiliation{
\institution{Texas A\&{}M University} 
\streetaddress{TAMU 3368}
\city{College Station}
\state{Texas}
\postcode{77843-3368} 
}

\begin{abstract}   
Suppose $F:=(f_1,\ldots,f_n)$ is a system of random $n$-variate polynomials 
with $f_i$ having degree $\leq\!d_i$ and the coefficient of
$x^{a_1}_1\cdots x^{a_n}_n$ in $f_i$ being an independent complex Gaussian of
mean $0$ and variance $\frac{d_i!}{a_1!\cdots a_n!\left(d_i-\sum^n_{j=1}a_j
\right)!}$. Recent progress on Smale's
17$\thth$ Problem by Lairez --- building upon seminal work of Shub, Beltran, 
Pardo, B\"{u}rgisser, and Cucker --- has resulted in a deterministic algorithm 
that finds a single (complex) approximate root of $F$ using just 
$N^{O(1)}$ arithmetic operations on average, where 
$N\!:=\!\sum^n_{i=1}\frac{(n+d_i)!}{n!d_i!}$\linebreak  
($=n(n+\max_i d_i)^{O(\min\{n,\max_i 
d_i)\}}$) is the maximum possible total number of monomial terms for such 
an $F$. However, can one go faster when the number of terms is smaller,  
and we restrict to real coefficient and real roots? 
And can one still maintain average-case polynomial-time with 
more general probability measures?  

We show the answer is yes when $F$ is instead a binomial system ---   
a case whose numerical solution is a key step in polyhedral homotopy 
algorithms for solving arbitrary polynomial systems.  
We give a deterministic 
algorithm that finds a real approximate root (or correctly decides there are 
none) using just
$O(n^2(\log(n)+\log\max_i d_i))$ arithmetic operations on average. 
Furthermore, our approach allows Gaussians with arbitrary variance.
We also discuss briefly the obstructions to maintaining 
average-case time polynomial in $n\log \max_i d_i$ when $F$ has more terms. 
\end{abstract} 

\ccsdesc[500]{Theory of computation~Algebraic complexity theory}

\keywords{Smale's 17th Problem, real roots, sparse polynomial, 
Newton iteration, approximate root } 

\maketitle 

\section{Introduction}  
Polynomial system solving has occupied a good portion of research 
in algebraic geometry for centuries, and inspired numerous algorithms 
in engineering and optimization. In recent years, {\em homotopy continuation}   
(see, e.g., \cite{ms,lw91,li97,sw,bhsw}) 
has emerged as one of the most practical and efficient approaches to leverage  
high performance computing for the approximation of roots of large polynomial 
systems. A refinement particularly useful for sparse systems is 
{\em polyhedral homotopy} \cite{hs,verschelde,leeli}. To be brutally 
concise, polyhedral homotopy reduces the solution of an arbitrary polynomial 
system to (a) solving a finite collection of {\em binomial} systems to 
high precision and then (b) iterating a finite collection of rational 
functions. 

It is thus important to have rigorous and, ideally, optimal 
complexity estimates for solving binomial systems. 
Since solving arbitrary polynomial systems is a numerical problem involving 
solutions of unknown minimal spacing, we will need to incorporate the 
cost of approximating well enough to distinguish distinct solutions. 
A recent and elegant way to handle this is via the notion of {\em approximate 
root in the sense of Smale}. In what follows, we use $|\cdot|$ for the 
standard $\ell_2$-norm on $\Cn$. 
\begin{dfn} \cite{smale,bcss} 
Given any analytic function $F : \Cn \longrightarrow \Cn$, we 
define the {\em Newton endomorphism of $F$} to be 
$N_F(z)\!:=\!z-F'(z)^{-1}F(z)$, where we think of $F(z)$ as a 
column vector and we identify the derivative $F'(z)$ with the 
matrix of partial derivatives $\left. \left[\frac{\partial f_i}{\partial x_j}
\right]\right|_{x=z}$. We call $\zeta\!\in\!\Cn$ a {\em non-degenerate} root of 
$F$ if and only if $F'(\zeta)$ is invertible. Given  
$z_0\!\in\!\Cn$, we then define its {\em sequence of Newton iterates} 
$(z_n)_{n\in\N\cup\{0\}}$ via the recurrence $z_{n+1}\!:=\!N_F(z_n)$ 
(for all $n\!\geq\!0$). We then call $z_0$ {\em an approximate root of 
$F$ in the sense of Smale (with associated true root $\zeta$)} if and only 
if $F$ has a 
non-degenerate root $\zeta\!\in\!\Cn$ satisfying $|z_n-\zeta|\!
\leq\!\left(\frac{1}{2}\right)^{2^{n-1}}|z_0-\zeta|$ for all 
$n\!\geq\!1$. \dia 
\end{dfn} 

\noindent 
In essence, once one has an approximate root in the sense above, one can 
easily compute coordinates within any desired $\eps\!>\!0$ of the 
coordinates of a {\em true} root,  
simply by computing $O\!\left(\log \log \frac{1}{\eps}\right)$ 
Newton iterates. The special case $F(z_1)\!:=\!z^2_1-2$ already 
shows that one needs $\Omega\!\left(\log \log \frac{1}{\eps}\right)$ 
arithmetic operations to compute $\sqrt{2}$ within $\eps$ \cite{bshouty}. 
So one can arguably consider an approximate root to be the gold standard 
for specifying a true root. In particular, one no longer has to worry 
about finding the minimal root spacing of $F$ (to get the right 
$\eps$ for approximations within $\eps$), since an approximate 
root in the sense of Smale is guaranteed to converge optimally fast to a 
unique true root. 

Of course, this begs the question of how one can possibly find an 
approximate root. This is the crux of Smale's 17$\thth$ Problem 
(see \cite{21a,21b} and Section \ref{sub:s17} below), 
which was recently positively solved 
by Lairez \cite{lairez}. (See also the seminal work of Beltran and Shub 
\cite{bezout6,bezout7}, Beltran and Pardo 
\cite{bps17a,bps17b,bps17c} and B\"{u}rgisser and Cucker \cite{bc17}.) 
Roughly, Lairez's discovery was an algorithm that,  
for a certain class of {\em random} polynomial systems, finds a single 
(complex) approximate root 
in polynomial-time on average. We now introduce some more terminology 
to be precise: 
\begin{dfn} 
Suppose $\cA_1,\ldots,\cA_n\!\subset\!\Zn$ are finite 
subsets and $\{c_{i,a} \; | \; i\!\in\!\{1,\ldots,n\} \text{ and } 
a\!\in\!\cA_i \text{ for all } i\}$ is a collection of independent 
complex Gaussians with mean $0$ and the variance of $c_{i,a}$ equal to 
$v_{i,a}$. Letting $a\!:=\!(a_1,\ldots,a_n)$, 
$x^a\!:=\!x^{a_1}_1\cdots x^{a_n}_n$, and  
$f_i(x)\!:=\!\sum_{a \in \cA_i} c_{i,a} x^a$, we call 
$F\!:=\!(f_1,\ldots,f_n)$ an {\em $n\times n$ random polynomial 
system with support $(\cA_1,\ldots,\cA_n)$}. \dia 
\end{dfn}  
\begin{lai} \cite[Thm.\ 23]{lairez}\footnote{We have paraphrased a bit: 
Lairez's main theorem is stated in terms of homogeneous polynomials, 
and he counts square roots as arithmetic operations as well. 
Via the techniques of, say,  
\cite{bps17b}, one can easily derive our affine statement.}  
Following the notation above, let $d_1,\ldots,d_n\!\in\!\N$, 
$\cA_i\!:=\!\{(a_1,\ldots,a_n)\!\in\!(\N\cup\{0\})^n\; | \; 
\sum^n_{j=1} a_j\!\leq\!d_i\}$ for all $i$, and 
$v_{i,a}\!:=\!\frac{d_i!}{a_1!\cdots a_n! \left(d_i-\sum^n_{j=1}a_j\right)!}$. 
Then one can find a (complex) approximate root of $F$ using 
just $O(nd^{3/2}N(N+n^3))$ arithmetic operations on average, where  
$N\!:=\!\sum^n_{i=1}\frac{(d_i+n)!}{d_i!n!}$ and $d\!:=\!\max_i d_i$. \qed  
\end{lai}  

\noindent 
Note that restricting the support $(\cA_1,\ldots,\cA_n)$ is a 
way to consider {\em sparsity} for one's polynomial system. 
In particular, one can think of Lairez's Theorem as solving 
Smale's 17$\thth$ Problem in the ``dense'' case, since Lairez assumes 
that {\em all} monomial terms up to a given degree appear (with probability 
$1$) in each polynomial $f_i$. Indeed, one should note that Smale never 
specified what kind of probability measure one should use 
in his 17$\thth$ Problem \cite{21a,21b}. So Smale's 17$\thth$ Problem 
actually includes sparse systems if some of the 
random coefficients have mean, and all higher moments, equal to $0$.  
Smale also observed that one can pose a more difficult analogue of 
his 17$\thth$ problem over the real numbers. 

Observe that $\sum^n_{i=1}\frac{(d_i+n)!}{d_i!n!}$ is exactly 
the maximal possible total number of monomial terms in an $n\times n$ 
polynomial 
system where $f_i$ has degree $d_i$. Note also that just evaluating 
a monomial of degree $d$ takes $\Omega(\log d)$ arithmetic operations: 
Simply consider the straight-line program complexity of the integer 
$2^d$ (see, e.g., \cite{brauer,svaiter,moreira}). One should pay attention to 
the evaluation complexity of $F$ since Lairez's algorithm uses  
Newton iteration, which in turn requires evaluating $F$ (and $F'$) 
many times. So one can then naturally ask, in the spirit of real fewnomial 
theory \cite{kho}: Can one find a real approximate root of $F$ (or decide 
whether there is no real root) using, say, $(t\log d)^{O(1)}$ arithmetic 
operations on average, when $t$ is the total number of monomial terms of $F$ 
and $d\!:=\!\max_i d_i$? This would be a significant new speed-up. For 
instance, the special case $t\!=\!O(n)$ is already  
quite non-trivial since there are standard algebraic tricks (e.g., 
the bottom of the first page of \cite{es}) to reduce arbitrary 
polynomial systems to trinomial systems.   

Our first main theorem thus solves a special case of a 
refined version of Smale's 17$\thth$ Problem, and serves as a starting point 
for a deeper study of the randomized complexity of solving arbitrary real 
sparse polynomial systems. 
\begin{thm} 
\label{thm:first} 
Suppose $A\!=\![a_{i,j}]\!\in\!\Z^{n\times n}$ has  
nonzero determinant, and all the entries of $A$ have absolute value 
at most $d$. Suppose also that $c_{i,j}$ 
is an independent real Gaussian with mean $0$ and fixed 
(but otherwise arbitrary) variance, for each  
$(i,j)\!\in\!\{1,\ldots,n\} \times \{0,1\}$. Let $F\!:=\!(f_1,\ldots,f_n)$ 
with $f_i(x)\!:=\!c_{i,0}+c_{i,1}\cdot x^{a_{1,1}}_1\cdots x^{a_{1,n}}_n$. 
Then, on average, one can find a real approximate root of $F$ (or 
correctly determine there are no real roots) using just 
$O(n^2\log(nd))$ arithmetic operations 
and $O(n^{\omega+1}\log^2(dn))$ bit operations, where $\omega$ 
is any upper bound on the matrix multiplication exponent.   
\end{thm} 

\noindent 
We prove Theorem \ref{thm:first} in Section \ref{sec:proof}. 
The best current upper bound on $\omega$, as of January 2019, 
is $2.372873$ \cite{vaswil}.  
A fundamental ingredient behind our proof of 
Theorem \ref{thm:first} is a hybrid algorithm of Ye enabling 
the quick approximation of rational powers of a real number 
\cite{ye}, combined with some classic results on fast linear 
algebra over $\Z$ \cite{smith,storjophd}. 
A final key ingredient is estimating  
the expected value of linear combinations 
of logarithms of absolute values of standard real Gaussians 
(Proposition \ref{prop:Q3-answer} in Section \ref{sub:final} below). We 
were unable to find any explicit asymptotics for such expectations, so we 
derive these from scratch in the latter half of Section \ref{sec:back} 
and Section \ref{sec:prob}. 

We will explain some of the subtleties behind extending Theorem 
\ref{thm:first} to systems with arbitrary supports in Section \ref{sub:subtle} 
below. First, however, let us briefly review the original statement of 
Smale's 17$\thth$ Problem. 

\subsection{Quick Review of Smale's 17$\thth$ Problem} 
\label{sub:s17} 
Smale's 17$\thth$ Problem \cite{21a,21b} elegantly summarizes the subtleties 
behind polynomial system solving: 
\begin{quote} 
Can {\bf a zero} of $n$ complex polynomial equations 
in $n$ unknowns be {\bf found approximately}, 
{\bf on the average}, in polynomial-time with a uniform algorithm? 
[Emphases added.] 
\end{quote}    

We clarify the notion of ``polynomial-time'' below. As 
motivation, let us first see how the emphasized  
terms highlight fundamental difficulties in polynomial system solving: 
\begin{itemize}
\item[]{\mbox{}\hspace{-1cm}{\bf ``a zero'':} We can not expect a fast 
algorithm approximating 
{\em all} the roots since, for $n\!\geq\!2$, there may be infinitely many. 
In which case, for $d_1\!\geq\!3$ (e.g., the case of elliptic curves 
\cite{silvermantate}), the roots 
will likely not admit a rational parametrization. When there are only finitely 
many roots, systems like\linebreak 
$(x^2_1-1,\ldots,x^2_n-1)$ show that the number of 
roots can be exponential in $n$.} 
\item[]{\mbox{}\hspace{-1cm}{\bf ``found approximately'':}  
Even restricting to integer 
coefficients, the number of digits of accuracy needed to separate distinct 
roots can be exponential in $n$, e.g.,\\ 
\mbox{}\hfill $((2x_1-1)(3x_1-1),x_2-x^2_1,
\ldots,x_n-x^2_{n-1})$\hfill\mbox{}\\ 
has roots with $n\thth$ coordinates $\frac{1}{2^{2^{n-1}}}$ and 
$\frac{1}{3^{2^{n-1}}}$. So, especially for irrational coefficients, 
we need a more robust notion of approximation than digits of accuracy. 
(Hence's Smale's definition of approximate root from \cite{smale}.) }  
\item[]{\mbox{}\hspace{-1cm}{\bf ``on the average'':} Restricting to integer 
coefficients, distinguishing between a system having finitely many or 
infinitely many roots is $\np$-hard (see, e.g., \cite{plaisted,koidim}). 
Furthermore, as already long known in the numerical linear algebra community 
(e.g., results on the distribution of eigenvalues of random matrices 
\cite{edelman,taovu}), even if the number of roots is finite, 
the accuracy needed to separate distinct roots can vary wildly as a function of 
the coefficients. So averaging over all inputs allows us to amortize  
the complexity of potentially intractable instances. } 
\end{itemize} 

The original statement of Smale's 17$\thth$ Problem measures {\em time} 
(or {\em complexity}) as the total number of (a) (exact) field operations 
over $\R$, (b) comparisons over $\R$, and (c) bit operations 
\cite{21a}. (The underlying computational model is a {\em 
BSS machine over $\R$} \cite{bcss}, which is essentially a classical {\em 
Turing} machine \cite{papa,arorabarak,sipser}, augmented so 
that it can perform any field operation or comparison over $\R$ in one time 
step.) {\em Polynomial-time} was then meant as polynomial in the 
number of (nonzero) coefficients of $F$. Smale interpreted the number 
of coefficients (which can be as high as $\sum^n_{i=1}\binom{d_i+n}{n}$ for 
$F$ as specified above) as the {\em input size}. 

\begin{rem} 
\label{rem:random} 
The precise probability distribution over which 
one averages was never specified in Smale's original statement \cite{21a,21b}. 
In all the literature so far on the problem 
(see, e.g., \cite{21a,21b,bps17a,bps17b,bezout6,bezout7,bps17c,bc17,lairez}), 
the {\em Bombieri-Weyl measure} was used: This is the choice of variances  
involving multinomial coefficients written earlier. \dia 
\end{rem} 

\noindent 
While the Bombieri-Weyl measure satisfies some very nice group invariance 
properties (see, e.g., \cite{kostlan,ss2,shiffman,lerario}), 
there is currently no widely-accepted notion of a 
``natural'' probability distribution for a random polynomial. For instance, 
there are several different distributions of interest already for the matrix 
eigenvalue problem (see, e.g., \cite{edelman,rojasutah,oxford}). More to the 
point, much work has gone into finding useful properties of the roots of 
random polynomials that are distribution independent (see, e.g., 
\cite{bsrandom,taovu2}). 

The meaning of {\em uniform algorithm} is more technical and is 
formalized in \cite{bcss} (see also \cite{papa,arorabarak,sipser} for the 
classical Turing case). Roughy, uniformity refers to having an implementation 
that can handle all input sizes, as opposed to having a different 
implementation for each input size. 

\subsection{Current Obstructions to Fully Incorporating Sparsity}  
\label{sub:subtle} 
As we'll see from the proof of our main theorem, solving 
an $n\times n$ system of Gaussian random binomials of degree 
$d$ can be reduced to solving $n$ univariate binomials 
of degree $(nd)^{O(n)}$, where the underlying coefficients are no 
longer Gaussian but have reasonably estimable means. Algebraically, 
this will imply that the underlying field extension (where one adjoins the 
coordinates of the solutions to the field generated by the coefficients) 
is always a radical extension.  

A natural next step then is to consider {\em $n\times n$ 
unmixed $(n+1)$-nomial systems}:\\  
$(c_{1,0}+c_{1,1}x^{a_1}+\cdots+c_{1,n}x^{a_n},
\ldots,c_{n,0}+c_{n,1}x^{a_1}+\cdots+c_{n,n}x^{a_n})$, 
where $a_i\!:=\!(a_{1,i},\ldots,a_{n,i})$ for all $i$.    
Via Gauss-Jordan Elimination, one can reduce such a system 
to a binomial system without affecting the roots. Unfortunately, if one 
starts with a system of the form above, with Gaussian $c_{i,j}$, 
the resulting binomial system no longer has Gaussian 
coefficients. So one needs to consider binomial systems with 
coefficient distributions more general than Gaussian, and 
we do this in a sequel to this paper.  

Going a bit farther, $n\times n$ unmixed $(n+2)$-nomial 
systems yield an interesting complication: The underlying 
field extensions need no longer be radical, even if $n\!=\!1$. 
A simple example is $x^5_1-2x_1+10$, which has Galois group 
$S_5$ over $\Q$. However, earlier results from \cite{rojasye} 
indicate that it should be possible to find real approximate roots 
quickly on average, at least for univariate trinomials. (One 
should also observe Sagraloff's recent dramatic speed-ups 
for the worst-case arithmetic complexity of approximating 
real roots of univariate sparse polynomials \cite{sagra}.) We 
conjecture that finding a real approximate root (or determining 
that there are no real roots) for a real Gaussian 
$n\times n$ unmixed $(n+2)$-nomial system is still possible 
in time $(n\log d)^{O(1)}$ on average, and hope to address 
this problem in the future. 

\section{Background}  
\label{sec:back} 
In what follows, for any $n\times n$ matrix $A\!\in\!\Z^{n\times n}$,
we define $x^A$ to be the vector of monomials
$\left(x^{a_{1,1}}_1\cdots x^{a_{n,1}}_n,\ldots,x^{a_{1,n}}_1
\cdots x^{a_{n,n}}_n\right)$. We call the substitution $x\!=\!z^A$ 
a {\em monomial change of variables}. The following proposition is 
elementary.  
\begin{prop}
\label{prop:mono}
We have that $x^{AB}\!=\!(x^A)^B$ for any $A,B\!\in\!\Z^{n\times n}$.
Also, for any field $K$, the map defined by $m(x)\!=\!x^U$,
for any unimodular matrix $U\!\in\!\Z^{n\times n}$, is an automorphism of
$(K^*)^n$. \qed
\end{prop}

Our main approach to solving binomial systems is to reduce 
them to systems of the form $(x^{d_1}_1-c_1,\ldots,x^{d_n}_n-c_n)$ 
via a monomial change of variables, and then prove that the 
distortion of the $c_i$ resulting from perturbing the original 
coefficients is controllable. Later on, we will also detail 
how a Gaussian distribution on the original coefficients implies 
that the $c_i$ still have well-behaved distributions. But now 
we will focus on quantifying our monomial changes of variables. 

\subsection{Linear Algebra Over $\Z$} 
\begin{dfn}
\label{dfn:smith}
Let $\gln(\Z)$ denote the
set of all matrices in $\Z^{n\times n}$ with determinant $\pm 1$
(the set of {\em unimodular} matrices).
Given any $M\!\in\!\Z^{n\times n}$, we call any 
identity of the form $UMV\!=\!S$ with $U,V\!\in\!\gln(\Z)$
and $S$ diagonal a {\em Smith factorization}. In particular, if
$S\!=\![s_{i,j}]$ and we
require additionally that $s_{i,i}\!\geq\!0$ and $s_{i,i}|s_{i+1,i+1}$
for all $i\!\in\!\{1,\ldots,n\}$ (setting $s_{n+1,n+1}\!:=\!0$), then $S$
is uniquely determined and is called {\em \underline{the} Smith
normal form} of $M$. \dia 
\end{dfn}
\begin{thm}
\label{thm:unimod}
\cite[Ch.\ 6 \& 8, pg.\ 128]{storjophd}
For any $A\!=\![a_{i,j}]\!\in\!\Z^{n\times n}$, a Smith
factorization of $A$ yielding the Smith normal form of $A$ 
can be computed within\\ 
\mbox{}\hfill $O\!\left(n^{\omega+1} 
\log^2(n\max_{i,j}|a_{i,j}|)\right)$\hfill\mbox{}\\ 
bit operations.
Furthermore, the entries of all matrices in the underlying 
factorization have bit size $O(n\log(n\max_{i,j}|a_{i,j}|))$. \qed
\end{thm}

\subsection{From Approximate Roots of Univariate Binomials to Systems} 
We begin with an important observation from 
the middle author's doctoral dissertation, building upon 
earlier work of Smale \cite{smale} and Ye \cite{ye}. 
\begin{lemma} \cite[Thm.\ 4.10]{kaitlyn} 
\label{lemma:uni} 
Let $d\!\in\!\N$ satisfy $d\!\geq\!2$,  
$c\!>\!0$, and $f(x_1)\!:=\!x^d_1-c$. Then we can 
find an approximate root of $f$ using 
$O\!\left(\log(d)+\log\log\max\{c,c^{-1}\}\right)$ 
field operations over $\R$. \qed 
\end{lemma} 

Since a monomial change of variables enables us to 
replace an arbitrary binomial system by a simpler, 
{\em diagonal} system of univariate binomials, it's 
enough to bound how the coefficients are distorted 
under such a change of variables. The following 
lemma gives us the bounds we need. 
\begin{lemma}  
\label{lemma:distort} 
Suppose $c_1,\ldots,c_n\!\in\!\Cs$ and 
$A\!\in\!\Z^{n\times n}$ has columns $a_1,\ldots,a_n$ 
and entries of absolute value at most $d$. Also 
let $\sigma\!:=\!\max_i\{|\log|c_i||\}$, 
let $UAV\!=\!S$ be the Smith Factorization of $A$, and 
let $(\gamma_1,\ldots,\gamma_n)\!:=\!(c_1,\ldots,c_n)^V$. 
Then the following bounds hold: \\ 
\mbox{}\hspace{.5cm}1. $\max_i|\log|\gamma_i||\!\leq\!n^{4+3n/2}d^{3n}
                         \sigma$. \\ 
\mbox{}\hspace{.5cm}2. If $\zeta\!=\!(\zeta_1,\ldots,\zeta_n)\!\in\!\Csn$ 
is a true root of $F$ then\\ 
\mbox{}\hspace{.8cm}$\max_i|\log|\zeta_i||\!\leq\!n^{O(n)}d^{O(n)}\sigma$. 
\hfill \qed 
\end{lemma} 

\subsection{Logs of Absolute Values of Gaussians} 
We now finally address the change in probability distribution resulting 
from replacing a Gaussian coefficient by a monomial in several other 
Gaussians. Our derivation is, necessarily, a bit long. So the hurried reader 
can jump to Propositions \ref{prop:lin} and \ref{prop:Q3-answer}, 
respectively in Sections \ref{sec:prob} and \ref{sub:final} below. 

Let $ \Theta, \Theta_{1}, \Theta_{2}, \ldots $ 
be independent exponential random variables, i.e.,  
$F_{\Theta} (t) := \P ( \Theta \leq t ) = 1- e^{-t}$.  
Let $ L, L_{1}, L_{2}, \ldots $ be independent symmetric exponential random 
variables, i.e., the density of $L$ is given by 
$\rho_{L} ( t)= \frac{1}{ 2} e^{ - |t|} , \ -\infty < t < \infty$,  
and similarly for the $L_i$. 

Let $Z, Z_{1}, Z_{2}, \ldots  $ be independent standard real 
Gaussian random variables, i.e.,  
\begin{equation}
\label{Z-1}
\Phi (t) := \P ( Z\leq t ) := \frac{1}{ \sqrt{2\pi} } \int_{-\infty}^{t} e^{ -\frac{ s^{2}}{ 2} }  ds= :\int_{-\infty}^{t} \phi(s) ds  . 
\end{equation}
Let $ Y, Y_{1},  Y_{2}, \ldots $ be independent random variables such that $ Y_{i} := \log|Z_{i}|$. 
We have that 
$$ F_{Y} (t) :=\P ( Y\leq t ) = \P( |Z|\leq e^{t} ) = \P ( -e^{-t} \leq Z\leq e^{t} ) = 1- 2 \Phi ( -e^{t} ) .$$
Taking derivatives we get $ F_{Y}^{\prime} (t):= 2e^{t} \phi(e^t)$, 
which implies that the density of $Y$, $\rho_{Y}$, is 
\begin{equation}
\label{f-Y}
\rho_{Y} (t) := \sqrt{\frac{2}{\pi}} e^{t}  e^{ - \frac{ e^{2t}}{2}} , \ -\infty <t <\infty. 
\end{equation}
We use $\E$ to denote expectation, $\P$ to denote probability, and define 
\begin{equation}
\label{a}
a:= \mathbb E Y :=\sqrt{\frac{2}{\pi}}\int_{-\infty}^{\infty} t e^{t} e^{-\frac{e^{2t}}{2}} d t  . 
\end{equation}
Note that 
\begin{equation}
\label{a-2}
0<a<e\sqrt{\frac{2}{\pi}} +  2<5.
\end{equation}
Indeed, $ a=\sqrt{\frac{2}{\pi}} \left( \int_{0}^{\infty} t e^{-t} e^{ - \frac{ e^{-2t}}{2}} d t + \int_{0}^{\infty} t e^{t} e^{ - \frac{ e^{2t}}{2}} d t\right)$. 
So clearly $a>0$ and also 
$$ \sqrt{\frac{\pi}{2}} a \leq \int_{0}^{\infty} t e^{-t} d t  + \int_{0}^{1} t e^{t} d t + \int_{1}^{\infty} t e^{t} e^{ - \frac{ e^{2t}}{2}} d t $$
$$ \leq 1+ (e-1) + \int_{e}^{\infty} \log{s} e^{-\frac{s^{2}}{2}} d s \leq e + \sqrt{2\pi}  \int_{e}^{\infty} s^{2}\phi(s) d s \leq e + \sqrt{2\pi} . $$
We define a new, centered (i.e., mean $0$) random variable via  
\begin{equation}
W:= Y- a .
\end{equation}

\noindent We write $A\simeq B$ to indicate that there exist positive 
constants $c_{1}, c_{2}$ with $ c_{1} A \leq B \leq c_{2} A$.  Let 
$ {\bf a} := ( a_{1} , \ldots , a_{k} )\!\in\!\R^k $ and define
\begin{equation}
\label{W} 
W_{{\bf a}} := \sum_{i=1}^{k} a_{i} Y_{i}  ,  \ V_{{\bf a}} := e^{W_{{\bf a}} }  \ {\rm and }  \ X_{{\bf a}} := \max\{ V_{{\bf a}} , V_{{\bf a}}^{-1}\} 
\end{equation}

Using the notation $\|R\|_p:=E(R^p)^{1/p}$, we will prove the following fact: 
\begin{lemma}
\label{moments-W}
\noindent Let $W$ be the centered random variable defined in \eqref{W} and 
let $p\geq 2$. Then
\begin{equation}
\| W\|_{p} \simeq \| \Theta \|_{p} \simeq \|L\|_{p} .
\end{equation}
\end{lemma}

\begin{proof}
We have that 
$$ \| W\|_{p}^{p} = \sqrt{\frac{2}{\pi}} \int_{-\infty}^{\infty} | t- a|^{p} e^{t} e^{-\frac{e^{2t}}{2}} d t =$$
$$ \sqrt{\frac{2}{\pi}}  \int_{0}^{\infty} |t+a|^{p} e^{-t} e^{- \frac{e^{-2t}}{2}} dt + \sqrt{\frac{2}{\pi}}  \int_{0}^{\infty} |t-a|^{p} e^{t} e^{- \frac{e^{2t}}{2}} dt  = :  \sqrt{\frac{2}{\pi}} I_{1} + \sqrt{\frac{2}{\pi}}  I_{2} .$$
Note that for $t\geq 1$, $ e^{ -\frac{e^{-2t}}{2}} \geq e^{-\frac{t}{2}}$. 
So using the above we have that 
$$ I_{1} \geq \int_{1}^{\infty} | t + a|^{p} e^{-t}  e^{ -\frac{e^{-2t}}{2}} dt \geq \int_{1}^{\infty} t^{p} e^{-\frac{3t}{2}} dt = \left( \frac{2}{3}\right)^{p+1} \int_{3/2}^{\infty} s^{p}e^{-s} d s =$$
$$ \left( \frac{2}{3}\right)^{p+1} \left( (p+1)! - \int_{0}^{3/2} s^{p} e^{-s}ds\right) \geq \frac{ \| \Theta \|_{p}^{p}}{ 4^{p}},  $$
since $ \| \Theta \|_{p}^{p}= (p+1)! $. Moreover 
$$ I_{1} \leq \int_{0}^{\infty} |t+a|^{p} e^{-t} d t = \| \Theta + a\|_{p}^{p} \leq 2 \| \Theta\|_{p}^{p} , $$
by Minkowski inequality the fact that $p\geq 2$ and \eqref{a-2}. So we have shown that 
\begin{equation}
\label{I1}
 \frac{\| \Theta\|_{p}^{p}}{ 4^{p}} \leq I_{1} \leq \| \Theta\|_{p}^{p} . 
\end{equation}
Moreover, using again \eqref{a-2}, 
$$ I_{2} \leq \int_{0}^{1} |t-a|^{p} e^{t} e^{ - \frac{ e^{2t}}{2}} d t +  \int_{1}^{\infty} |t-a|^{p} e^{t} e^{ - \frac{ e^{2t}}{2}} d t $$
$$ \leq e 5^{p} + \sqrt{2\pi}\int_{e}^{\infty} | \log{s}- a |^{p} \phi(s) d s \leq 5^{p} \left( e+ \| Z\|_{p}^{p}\right) \leq 5^{p} \|\Theta\|_{p}^{p} .$$
So we have shown that 
\begin{equation}
\label{I2}
 0 \leq I_{2} \leq   5^{p}\|\Theta\|_{p}^{p} . 
\end{equation}
Combining \eqref{I1} and \eqref{I2} we get that $ \frac{ \|\Theta\|_{p}}{4}\leq \| W\|_{p} \leq 6 \| \Theta \|_{p}$.
Finally, it is straightforward to check that $ \| \Theta \|_{p} \simeq \|L\|_{p}$ for all $p>0$. \end{proof}

\subsection{A Tool for Linear Combinations of Logs of Absolute Values of 
Gaussians}
\noindent We are going to use the following fundamental result of 
Latala: 
\begin{thm} 
\label{Latala}
\cite[Thm.\ 2 \& Rem.\ 2]{L}  
Let $ X_{1}, \ldots, X_{n}$ be centered independent random variables and 
$p\geq 2$. Then 
\begin{equation}
\left\| \sum_{i=1}^{n} X_{i} \right\|_{p} \simeq \|| (X_{1}, \ldots , X_{n})  \||_{p} ,
\end{equation}
where $ \|| (X_{1}, \ldots , X_{n})\||_{p}$ is defined to be\\  
\mbox{}\hfill 
$\inf\left\{t>0:  \sum_{i=1}^{n} \log \left(\mathbb E\frac{ \left| 
\frac{ X_{i} } {t}+1\right|^p+ \left|\frac{-X_{i}}{t} + 1\right|^p}
{2} \right)\leq p  \right\}$.
\hfill\qed 
\end{thm}

We will also need the following fact: 
\begin{lemma}
\label{p-comp}
\noindent Let $ X_{1}, \ldots , X_{n}$ be independent random variables and let $ \tilde{X}_{1}, \ldots , \tilde{X}_{n}$ be another sequence of independent random variables.  Fix $p\geq 2$ be an even integer and assume that there are 
$a,b>0$ such that 
\begin{equation}
\label{p-comp-assum}
a\| X_{i}\|_{q} \leq \| \tilde{X}_{i} \|_{q} \leq b\|X_{i} \|_{q}
\end{equation}
for all $1\leq q\leq p$ and for all $1\leq i \leq n$.  Then we have that 
\begin{equation}
\label{p-comp-res}
 a\|| (X_{1}, \ldots , X_{n})\||_{p} \leq  \|| (\tilde{X}_{1}, \ldots , \tilde{X}_{n})\||_{p} \leq b  \|| (X_{1}, \ldots , X_{n})\||_{p}.
\end{equation}
\end{lemma}

\begin{proof}
\noindent We will first prove the following 

\noindent {\it Claim}: Under the assumptions of the Lemma we have that for every $t>0$
\begin{equation}
\label{claim}
 \mathbb E \eta\!\left(aX_{i}/t\right) \leq  \mathbb E \eta\!\left(\tilde{X}_{i}/t\right) \leq   \mathbb E \eta\left( bX_{i}/t\right) , \ 1\leq i \leq n ,
\end{equation}
where $ \eta(x):= \frac{1}{2} \left( | x+1|^{p} + |1-x|^{p}\right)$. 

Indeed, $ \mathbb E \eta\!\left( \tilde{X}_{i}/t\right)  =\frac{1}{2} \sum_{k=0}^{p} {p\choose k} \mathbb E \left( ( \tilde{X}_{i}/t)^{k}) +  ( -\tilde{X}_{i}/t)^{k})\right)$ 
$$ = \frac{1}{2} \sum_{k=0, k \ {\rm even} }^{p} {p\choose k} \mathbb E \left( ( \tilde{X}_{i}/t)^{k}) +  ( \tilde{X}_{i}/t)^{k})\right)$$ 
$$ \leq \frac{1}{2} \sum_{k=0, k \ {\rm even}}^{p} {p\choose k} \mathbb E \left( ( bX_{i}/t)^{k}) +  ( bX_{i}/t)^{k})\right) $$
$$ =\frac{1}{2} \sum_{k=0}^{p} {p\choose k} \mathbb E \left( (  bX_{i}/t)^{k}) +  ( - bX_{i}/t)^{k})\right) =  \mathbb E \eta\!\left(bX_{i}/t\right). $$
The proof of the other side inequality in \eqref{claim} is identical. 
Equation \eqref{p-comp-res} then follows immediately from the claim and the 
definition of $\|| (X_{1}, \ldots , X_{n})\||_{p} $. 
\end{proof}

Our preceding lemma leads to the following:  
\begin{cor}
\label{cor}
\noindent Let $ {\bf X}:= (X_{1}, \ldots , X_{n})$, ${\bf \tilde{X}}:= ( \tilde{X}_{1}, \ldots , \tilde{X}_{2}) $  be two centered random vectors with 
independent coordinates and let 
$\theta= (\theta_{1}, \ldots , \theta_{n}) \in \mathbb R^n$. We assume that 
\eqref{p-comp-assum} holds true.  Then for every $1\leq r \leq p$, 
\begin{equation}
\label{cor-1}
c_{1} a \| \langle {\bf X}, \theta \rangle \|_{r} \leq \| \langle {\bf \tilde{X}}, \theta \rangle \|_{r} \leq c_{2} b \| \langle {\bf X}, \theta \rangle \|_{r},
\end{equation}
where $c_{1}, c_{2}>0$ are universal constants. 
\end{cor}

\begin{proof}
The result follows from Theorem \ref{Latala} and Lemma \ref{p-comp} applied to the random variables $ \theta_{i} X_{i}$ and $ \theta_{i} \tilde{X_{i}}$. 
\end{proof}

\section{Additional Probabilistic Estimates}
\label{sec:prob} 
Let $ {\bf W}:= ( W_{1}, \ldots , W_{n})$ be the centered random vector with 
independent entries that are logs of absolute values of real standard 
Gaussians. Let $ {\bf L} := (L_{1}, \ldots , L_{n})$. 
Let $\theta \in S^{n-1}$ (Here $S^{n-1}$ is the unit sphere in dimension $n$.)  
The next theorem below is a special case of a more general result of Gluskin 
and Kwapien \cite{GK}. Let us introduce some notation. Let $ x\in \Rn$. We 
write $ x^{\ast}$ for the non-increasing rearrangement of the vector 
$ ( |x_{1}|, \ldots, | x_{n}| )$. Given any $1\leq s\leq n$ and a vector $x$ we 
denote $ x^{s} $ the vector with entries $ x_{i}^{\ast}$ for $i \leq s $ and 
$0$ otherwise and by $x_{s}$ the vector with entries $ 0$ for $i\leq s$ and 
entries $x^{\ast}$ for $i>s$. 

\begin{thm} (Special case of \cite{GK}) 
\label{th-L}
\noindent  There are constants $C_{1}, C_{2} >0$ such that for every $n\geq 1$, 
$p\geq 1$, and every $\theta \in S^{n-1}$, one has that 
\begin{equation}
\label{GK-1}
C_{1}p \| \theta^{p} \|_{\infty} + C_{1}\sqrt{p}  \|\theta_{p} \|_{2}  \leq  \| \langle L, \theta \rangle \|_{p } \leq C_{2}p \| \theta^{p} \|_{\infty} + C_{2}\sqrt{p}  \|\theta_{p} \|_{2}.  \text{ \qed } 
\end{equation}
\end{thm}

\noindent Lemma \ref{moments-W}, Theorem \ref{th-L}, and Corollary 
\ref{cor} together imply the following  

\begin{prop}
\label{pr-L}
There exists two constants $C_{1}, C_{2} >0$ such that for every $n\geq 1$, $p\geq 1$  and every $\theta \in S^{n-1}$, one has that 
\begin{equation}
\label{GK-2}
C_{1}p \| \theta^{p} \|_{\infty} + C_{1}\sqrt{p}  \|\theta_{p} \|_{2}  \leq  \| \langle W, \theta \rangle \|_{p } \leq C_{2}p \| \theta^{p} \|_{\infty} + C_{2}\sqrt{p}  \|\theta_{p} \|_{2} . 
\end{equation}
\end{prop}

\noindent The above result gives very precise estimates about the 
concentration of the function 
$$ \P\!\left(\left| \sum_{i=1}^{n} \theta_{i} \log|Z_{i}| 
- \sum_{i=1}^{n} \theta_{i} a \right| \geq t \right) $$ 
for all $t$. A less precise but simpler to use statement than Theorem 
\ref{th-L} is the following
estimate: For every $ \theta \in S^{n-1}$ 
\begin{equation}
\label{L-simple}
\P\!\left(\left| \sum_{i=1}^{n} \theta_{i} L_{i} \right| \geq t \right) 
\leq \exp\left\{-C \min\left\{ \frac{ t }{\|\theta\|_{\infty}}, t^{2}\right\}
\right\}, \ t>0
\end{equation}

\noindent Using the above we arrive at the following

\begin{prop}
\label{prop:lin} 
Let $Z_{1}, \ldots , Z_{n}$ be independent standard real Gaussian random 
variables, $\theta \in S^{n-1}$, and $ a$ as defined in \eqref{a}. Then the 
following holds:
\begin{equation}
\label{W-simple}
 \P\!\left( \left| \sum_{i=1}^{n} \theta_{i} \log|Z_{i}| - \sum_{i=1}^{n} 
\theta_{i} a \right| \geq t \right)  \leq C^{\prime}\exp\left\{-C \min\left\{ 
\frac{ t }{\|\theta\|_{\infty}}, t^{2}\right\}\right\}, 
\end{equation}
for $t\!>\!0$, where $C, C^{\prime}>0$ are absolute constants. 
\end{prop}

\begin{proof}
\noindent By \eqref{GK-2} we have that $ \|\langle W, \theta\rangle \|_{p} \leq C_{2}\sqrt{p} $ if $ p \leq \|\theta \|_{\infty}^{-2}$ and $ \|\langle W, \theta\rangle \|_{p} \leq C_{2}p \|\theta \|_{\infty}  $ otherwise. 
Using Markov's Inequality we get that 
$\P\!\left( |\langle W, \theta \rangle | \geq eC_{2} \sqrt{p}\right) 
\leq e^{-p}, \ {\rm if } \  p \leq \|\theta \|_{\infty}^{-2}$, 
or (if we will set $ e C_{2} \sqrt{p}= t$), for $ t\geq C_{3}$, 
\begin{equation}
\label{prop-1-1}
 \P\!\left( | \langle W, \theta \rangle | \geq t \right) \leq e^{ - C_{4} t } , \ {\rm if } \ t \leq \|\theta \|_{\infty}^{-1} 
\end{equation}
and $\P\!\left( | \langle W, \theta \rangle | \geq eC_{2} p \|\theta\|_{\infty}\right) \leq e^{-p}, \ {\rm if } \  p \geq \|\theta \|_{\infty}^{-2}$ 
or (if we will set $ e C_{2} p \|\theta \|_{\infty} = t$), for 
$ t\geq \frac{C_{4}}{ \|\theta\|_{\infty}} $,
\begin{equation}
\label{prop-1-2}
 \P\!\left( | \langle W, \theta \rangle | \geq t \right) \leq e^{ - C_{5} t } .
 \end{equation}
 Combining \eqref{prop-1-1} and \eqref{prop-1-2} and adjusting the constants properly  we get \eqref{L-simple}. 
\end{proof}

\subsection{On the Expectation of $\log\log$} 
\label{sub:loglog} 
Let $ Z, Z_{i}$ be independent real standard Gaussian random variables and 
let $d$ be a positive integer.

Let $ X:= \max\{ |Z| , |Z|^{-1} \}$. We have that 
\begin{equation}
\label{W-up-1}
\P\!\left( \log\log\{ e X \} \geq t \right) \leq  \sqrt{\frac{8}{\pi}} e^{ - (e^{t}-1)} , \ t\geq 0 .
\end{equation}
Indeed, for $ t\geq 0$, we have \\ 
$\P\!\left( \log\log\{ e X \} \geq t \right) = 
\P\!\left( X\geq e^{ e^{t}-1}\right) =  \P\!\left( |Z| \geq e^{ e^{t}-1}  
\ {\rm or} \ |Z|\leq e^{-(e^{t}-1)}  \right)$ 
$$ \leq \P\!\left( |Z| \geq e^{e^{t}-1} \right)  
 + \P\!\left(\ |Z| \leq e^{-(e^{t}-1)}  \right) 
 \leq 2 \P\! \left( \  |Z| \leq e^{-  (e^{t}-1)}  \right)$$ 
$\leq 4 \frac{1}{ \sqrt{2\pi} }   e^{-  (e^{t}-1)}$. 
So, we get that 
\begin{equation}
\label{W-up-2}
\E [ \log\log\{ e X \} ] \leq   \sqrt{\frac{8}{\pi}} .
\end{equation}
Indeed, since $ \log\log\{ e X\} \geq 0$, using \eqref{W-up-1},
$$ \mathbb E [ \log\log\{ e X \} ]  = \int_{0}^{\infty}  \P\! \left( 
\log\log\{ e X \} \geq t \right)  d t $$
$$ \leq \sqrt{\frac{8}{\pi}} \int_{0}^{\infty} 
 e^{ - (e^{t}-1) } d t = \sqrt{\frac{8}{\pi}} \int_{0}^{\infty} \frac{1}{ s+1} 
 e^{ -s } d s \leq   \sqrt{\frac{8}{\pi}} .$$
In what follows we assume that 
\begin{equation}
\label{assum-d}
d\geq e^{2}. 
\end{equation}
We will use the following elementary inequality: 
\begin{equation}
\label{el-ineq}
a+ b \leq 2 a b , \ a, b\geq 1 . 
\end{equation}
Since $ e X\geq e $ and $ d/e\geq e $, using \eqref{el-ineq} and 
\eqref{W-up-2} we get 
$$\mathbb E \log\log\{ d X\} = \mathbb E \log\{ \log(d/e) + \log\{ e X\}\}$$ 
$$\leq  \mathbb E \log \{ 2 (\log(d/e)) ( \log\{ e X\} ) \}$$
$$=\log{2} + \log\log( d/e) + E [ \log\log\{ e X \} ]  \leq \log{2} + \log\log( d/e) +  \sqrt{\frac{8}{\pi}} .$$
Moreover, since $ X\geq 1$, $ \log\log\{ d X \} \geq \log\log{d}$ and 
we conclude that   
\begin{equation}
\label{W-up-3}
 \log\log{d} \leq \mathbb E \log\log\{ d X\}  \leq \log{2} + \log\log( d/e) +  \sqrt{\frac{8}{\pi}}  . 
\end{equation}

\subsection{Log-Concavity}
\noindent A Borel measure $\mu$ in $ \mathbb R^{n}$ is called log-concave if for every compact sets $ A, B$ and $ \lambda \in (0,1)$ one has 
\begin{equation}
\label{log-concave}
\mu ( \lambda A + (1-\lambda) B ) \geq \mu( A)^{\lambda } \mu(B)^{1-\lambda} .
\end{equation}

\noindent 

\begin{thm}[Borell \cite{Borell}] 
\label{Borell}
Let $\mu$ be a Borel measure in $\mathbb R^{n}$ that gives positive mass to some open ball. Then $\mu$ is $log$-concave if and only if has a density 
$\rho_{\mu}$ that is a $log$-concave function i.e. $\rho_{\mu}$ is 
non-negative, supported on a convex set and 
$$ \rho_{\mu} ( \lambda x + (1-\lambda) y ) \geq \rho_{\mu}^{\lambda} (x) 
\rho_{\mu}^{1-\lambda} ( y) , \ x, y \in \mathbb R^{n} \ \lambda \in (0,1) . $$
\end{thm}

\begin{thm}[Pr\'ekopa]
\label{PL}
Sum of independent $log$-concave random variables is $log$-concave. 
\end{thm}

\begin{prop}
\label{borellslemma}
\noindent Let $ \mu$ be a $log$-concave probability measure and let $ K$ be a symmetric closed convex set in $\mathbb R^{n}$. Then if $ \delta := \mu( K) \geq \frac{1}{2}$ for every $ t>1$ we have that 
$$ \mu \left ( (tA)^{c} \right) \leq  \delta \left ( \frac{1- \delta }{ \delta } \right)^{ \frac{ t+1 }{ 2}} . $$
\end{prop}

\begin{cor}
\label{cor-Borell}
Let $X$ be a $log$-concave random variable with mean $0$ and variance $\gamma^{2}$. Then 
\begin{equation}
\P( |X| \geq s ) \leq  e^{ - \frac{s}{2\gamma}} , \ s\geq \gamma .
\end{equation}
\end{cor}
\begin{proof}
\noindent Let $ A:= \{ | x| \leq 2\gamma \}$. Then, by Chebychev's inequality we have that $\P ( A) = \delta \geq \frac{3}{4}$. By Proposition \ref{borellslemma} we get that 
$$ \P \left(  | X| \geq t \gamma \right) = \P \left( (tA)^{c} \right) \leq \delta \left( \frac{1- \delta}{ \delta } \right)^{\frac{t+1}{2}} \leq \left( \frac{1}{3}\right)^{\frac{t+1}{2}} \leq  e^{ - \frac{t}{2}}, \ t\geq 1 . $$ 
\end{proof}

\subsection{Final Estimates}
\label{sub:final} 
\noindent Recall that if $Z$ is a standard Gaussian then $ Y:= \log|Z|$ has density 
\begin{equation}
\label{density-Y} 
\rho_{Y}(t) := \sqrt{\frac{2}{\pi}}  e^{t}  e^{ - \frac{ e^{2t}}{t}} =: \sqrt{ \frac{2}{ \pi}} e^{ - v (t) } , \ -\infty < t < \infty.  
\end{equation}

Let $a:= \mathbb E [Y]$ and $ \tau^{2}$ be the variance of $ Y$. 

We have the following 
\begin{prop}
\label{prop-logconcavity}
\noindent Let $ {\bf a} \in \mathbb R^{k}$ and assume that $ \sum_{i=1}^{k} a_{i} = 0 $. Then $ W_{{\bf a}} $ is a $log$-concave random variable with expectation $0$ and variance $\gamma^{2} :=  \| {\bf a} \|_{2}^{2}\tau^{2}$. Then we have 
\begin{equation}
\label{pr-2.1-1}
 \P\! \left( \log\log{\{eX_{{\bf a}}\}}  \geq   t \right) \leq  e^{- \frac{ e^{t}-1}{ 2 \gamma}}, \ t \geq \log\{1+\gamma \}. 
\end{equation}
Moreover, 
\begin{equation}
\label{pr-2.1-2}
 \mathbb E [ \log\log{\{eX_{{\bf a}}\}} ]\leq 2 + \log\{ 1+\gamma\} .  
\end{equation}
\end{prop}

\begin{proof}

\noindent Note that $ v( t) :e^{ \frac{2t}{2}}- t $ is a convex function so by Borell's theorem $ Y$ is a log-concave random variable.  We have that 
$$ \mathbb E [W_{\bf a}] = \sum_{i=1}^{k} a_{i} \mathbb E[ Y_{i} ] = a \sum_{i=1}^{k} a_{i} = 0 $$ 
and since $ Y_{i}$ are independent 
$$ {\rm var} ( W_{{\bf a}} ) = \sum_{i=1}^{k} a_{i}^{2} {\rm var} ( Y_{i}) = \gamma^{2}  \sum_{i=1}^{k} a_{i}^{2} = \gamma^{2} \|{\bf a}\|_{2}^{2} . $$
So, we can estimate as follows:
$$ \P\! \left( \log\log{\{eX_{{\bf a}}\}}  \geq  t \right) = \P\! \left( X_{{\bf a}} \geq e^{ e^{t}-1}\right) =  \P\! \left( \{ V_{{\bf a}} \geq  e^{ e^{t}-1} \} \cup \{ V_{{\bf a}} \leq  e^{ -(e^{t}-1)} \}   \right) =$$
$$  \P\! \left( V_{{\bf a}} \geq  e^{ e^{t}-1}  \right) + \P\! \left(  V_{{\bf a}} \leq  e^{ -(e^{t}-1)}    \right) = \P \left( W_{{\bf a}} \geq   e^{t}-1  \right) + \P \left(  W_{{\bf a}} \leq   -(e^{t}-1)    \right) $$
$$ = \P \left( | W_{\bf a}| \geq e^{t}-1 \right) \leq  e^{- \frac{ e^{t}-1}{ 2 \gamma}}, $$
as long $ e^{t} - 1 \geq \gamma$, where we have also used Corollary \ref{cor-Borell}. Finally, since $ e X_{\bf a}\geq e $, $\log\log\{ X_{\bf a} \} \geq 0$, 
we have that 
$$  \mathbb E [ \log\log{\{eX_{{\bf a}}\}} ]\leq    \int_{0}^{\infty} \P \left( \log\log{\{eX_{{\bf a}}\}} \geq t \right) dt $$
$$ \leq \int_{0}^{ \log\{ 1+\gamma\} }  d t + \int_{ \log\{ 1+\gamma\} }^{\infty}   e^{- \frac{ e^{t}-1}{ 2 \gamma}} d t \leq \log\{ 1+\gamma\} + \int_{\gamma}^{\infty} \frac{1}{1+s} e^{ - \frac{ s}{ 2\gamma}} d s$$
$$ = \log\{ 1+\gamma\} + \int_{\frac{1}{2}}^{\infty} \frac{2\gamma}{1+2\gamma x } e^{ - x} d x$$
$$ \leq   \log\{ 1+\gamma\}  + \frac{ 2\gamma}{ 1+ \gamma} \int_{0}^{\infty} e^{-x} d x \leq 2 + \log\{ 1+\gamma\}.$$
\end{proof}

Working as in Subsection \ref{sub:loglog} we arrive at the following key 
result:  
\begin{prop}
\label{prop:Q3-answer}
\noindent Let $ {\bf a} \in \mathbb R^{k}$ and assume that $ \sum_{i=1}^{k} a_{i} = 0 $. We have that 
\begin{equation}
\label{pr-2.1-3}
 \log\log{d} \leq \mathbb E \log\log\{ d X_{\bf a}\}  \leq  \log\log( d/e) + 2 + \log{2} +\log\{ 1+ \tau\| a\|_{2} \} .
 \end{equation}
\end{prop}

\begin{proof}

Since $ e X\geq e $ and $ d/e\geq e $, using \eqref{el-ineq}, \eqref{pr-2.1-2} 
we get 
$$\mathbb E \log\log\{ d X_{\bf a}\} = \mathbb E \log\{ \log(d/e) 
 + \log\{ e X_{\bf a}\}  \} $$ 
$$ \leq  \mathbb E \log \{ 2 (\log(d/e)) ( \log\{ e X_{\bf a}\} ) \} 
= \log{2} + \log\log( d/e) + E [ \log\log\{ e X_{\bf a} \} ] $$ 
$$ \leq \log{2} + \log\log( d/e) + 2 + \log\{ 1+ \tau\| a\|_{2} \} .$$
Moreover, since $ X_{\bf a}\geq 1$, $ \log\log\{ d X_{\bf a} \} \geq \log\log{d}$ we get \eqref{pr-2.1-3}. 
\end{proof} 

\section{The Proof of Theorem \ref{thm:first}} 
\label{sec:proof} 
First note that the $c_{i,j}$ are all nonzero with probability $1$, 
so we may assume (since we are considering average-case complexity) 
that all the $c_{i,j}$ are nonzero. In which case, we can focus 
solely on roots in $\Rsn$. 

Now note that by Proposition \ref{prop:mono}, we can 
easily decide whether our input binomial system $F$ 
has a real root: If $F$ is diagonal, i.e., if 
$F\!=\!(c_{1,0}+c_{1,1}x^{d_1}_1,\ldots,c_{n,0}+c_{n,1}x^{d_n})$, 
then $F$ has a real root if and only if $c_{i,0}c_{i,1}\!<\!0$ 
for all $i$ with $d_i$ even. (In which case, each 
orthant of $\Rn$ contains at most $1$ root of $F$.)  
If $F$ is not diagonal, then 
after computing a Smith factorization $UAV\!=\!S$ (which accounts for 
our stated bit complexity bound, thanks to Theorem \ref{thm:unimod}), 
we can reduce to the diagonal case and simply check $n$ 
inequalities. If there are no real roots, no further work 
needs to be done. 

So let us now assume that there are real roots. Without 
loss of generality, we may assume there is a root in the 
positive orthant $\Rn_+$. This will be the root we will try to 
approximate. So we may now assume that 
we are trying to approximate the roots of $G\!:=\!(z^{s_{1,1}}_1-\gamma_1,
\ldots,z^{s_{n,n}}_n-\gamma_n)$ where \\ 
\mbox{}\hfill 
$(\gamma_1,\ldots,\gamma_n)\!:=\!(-c_{1,0}/c_{1,1},\ldots,-c_{n,0}/c_{n,1})^V$ 
\hfill\mbox{}\\ 
lies in $\Rn_+$, 
and the $s_{i,i}$ are the diagonal entries of the Smith normal form $S$ of 
$A$. In particular, we need to approximate the unique root $\mu$ of $G$ 
in $\Rn_+$ well enough so that $\zeta\!:=\!\mu^U$ is an approximate 
root of $F$. 

Thanks to Lemmata \ref{lemma:uni} and \ref{lemma:distort}, 
a quick derivative calculation tells us that it suffices 
to find an approximate root of $G$. (One needs some extra 
precision to ensure that $\zeta$ is an approximate root of $F$ 
but the bounds from Lemma \ref{lemma:distort} easily imply that 
the necessary extra work is negligible compared to our 
stated arithmetic complexity bound.) So it suffices to 
compute an upper bound on the expectation of 
$\sum^n_{i=1} 
\left[\log(|s_{i,i}|)+\log\log(e\max\{|\gamma_i|,|\gamma^{-1}_i|\})\right]$. 
We are almost done, save for the fact that the $\gamma_i$ are 
monomials in real Gaussians that need {\em not} have variance $1$. 

However, we can precede our construction of $G$ with {\em another} 
renormalization to reduce to the 
variance $1$ case: Observe that $x^A\!=\!c$ if and only if 
$(rx)^A\!=\!r^Ac$, for any $r\!\in\!\Rsn$. So if we take $r$ to be a 
suitable matrix power of a vector of ratios of variances, we 
can replace our original binomial system $F$ by a new binomial 
system $\tilde{F}$ with all coefficients being standard real Gaussians 
(and new root a rescaling of our old root). 
In particular, we merely take 
$r\!:=\!(v_{1,1}/v_{1,0},\ldots,v_{n,1}/v_{n,0})^{A^{-1}}$. 
Lemmata \ref{lemma:uni} and \ref{lemma:distort} once again imply 
that the cost of the necessary increase in precision to convert an 
approximate root of $\tilde{F}$ to an approximate root of $F$ 
is negligible. 

We now conclude via Proposition \ref{prop:Q3-answer} and Theorem 
\ref{thm:unimod}: Our desired expectation is at most\\ 
\ \ $\sum^n_{i=1}\left[n\log(nd)+
0+2+\log(2)+\log\left(1+\tau\sqrt{n}e^{O(n\log(nd))}\right)\right]$. \\ 
The last quantity is clearly $O(n^2\log(nd)+n\log(\sqrt{n}n\log(nd)))$ or, 
more simply, $O(n^2\log(nd))$. \qed 

\section*{Acknowledgements}  
We humbly thank REU students Caleb Bugg and Paula Burkhardt 
for important discussions on preliminary versions of this work. 
We also thank the NSF for their support through 
grants CCF-1409020, DMS-1460766, DMS-1757872, and 
DMS-1812240.  

\bibliographystyle{amsalpha}

\end{document}